\documentclass[11pt]{amsart}
\usepackage{amscd,amssymb,longtable, rotating, lscape, graphicx}
\usepackage[matrix,arrow,curve]{xy}
\usepackage{supertabular}
\usepackage{setspace}
\theoremstyle{definition}
\newtheorem{theorem}[equation]{Theorem}
\newtheorem*{theorem*}{Theorem}
\newtheorem{lemma}[equation]{Lemma}
\newtheorem{corollary}[equation]{Corollary}

\newtheorem{example}[equation]{Example}
\newtheorem{definition}[equation]{Definition}

\newtheorem*{definition*}{Definition}

\theoremstyle{remark}
\newtheorem{remark}[equation]{Remark}

\makeatletter\@addtoreset{equation}{section}
\makeatletter\@addtoreset{section}{part}

\makeatother

\def \P {\mathbb{P}}

\def \C {\mathbb{C}}
\def \Z {\mathbb{Z}}

\def \wt {\mathrm{wt}}

\def \le {\leqslant}
\def \ge {\geqslant}

\author{Ivan Cheltsov and Constantin Shramov}

\title{Del Pezzo zoo}

\thanks{The first author was supported by the~grants NSF DMS-0701465 and
EPSRC EP/E048412/1, the~second~author was supported by the~grants
RFFI No. 08-01-00395-a, N.Sh.-1987.2008.1 and EPSRC EP/E048412/1.}

\begin{document}

\begin{abstract}
We study del Pezzo surfaces that are  quasismooth and well-formed
weighted hypersurfaces. In particular, we find all such surfaces
whose~$\alpha$-in\-va\-ri\-ant~of~Tian~is~greater~than~$2/3$.
\end{abstract}

\maketitle

All varieties are  assumed to be complex,  projective and normal.

\section{Introduction}
\label{section:intro}

Let $X$ be a~hypersurface in $\mathbb{P}(a_{0},\ldots,a_{n})$ of
degree $d$, where $a_0\leqslant\ldots\leqslant a_n$.
Then $X$ is
given~by
$$
\phi\big(x_{0},\ldots,x_{n}\big)=0\subset\mathbb{P}\big(a_{0},\ldots,a_{n}\big)\cong\mathrm{Proj}\Big(\mathbb{C}\big[x_{0},\ldots,x_{n}\big]\Big),
$$
where $\mathrm{wt}(x_{i})=a_{i}$, and $\phi$ is a~quasihomogeneous
polynomial of degree $d$. The~equation
$$
\phi\big(x_{0},\ldots,x_{n}\big)=0\subset\mathbb{C}^{n+1}\cong\mathrm{Spec}\Big(\mathbb{C}\big[x_{0},\ldots,x_{n}\big]\Big),
$$
defines a~quasihomogeneous singularity $(V,O)$, where $O$ is
the~origin of $\mathbb{C}^{n+1}$.

\begin{definition}
\label{definition:quasismooth} The~hypersurface $X$ is quasismooth
if the~singularity $(V,O)$ is isolated.
\end{definition}

Suppose that $X$ is quasismooth.

\begin{remark}\label{remark:rational-vs-canonical-etc}
It follows from \cite[Theorem~7.9]{Ko97}, \cite[Proposition~8.13]{Ko97}
and~\cite[Remark~8.14.1]{Ko97} that $\sum_{i=0}^{n}a_{i}>d$
if and only if
the singularity $(V,O)$ is canonical. Moreover,
since $(V,O)$ is Gorenstein, it is canonical if and only if
it is rational (see~\cite[Theorem~11.1]{Ko97}).
\end{remark}

\begin{definition}
\label{definition:well-formed} The~hypersurface
$X\subset\mathbb{P}(a_{0},\ldots,a_{n})$ is well-formed if
$$
\mathrm{gcd}\big(a_0,\ldots,\widehat{a_i},\ldots,\widehat{a_j},\ldots,a_n\big)\mid d%
$$
and $\mathrm{gcd}(a_0,\ldots,\widehat{a_i},\ldots,a_n)=1$ for
every $i\ne j$.
\end{definition}

Suppose that $X$ is well-formed. Then $\sum_{i=0}^{n}a_{i}>d$
if and only if $X$ is a~Fano variety. Put
$$
I=\sum_{i=0}^{n}a_{i}-d,
$$
and suppose that $\sum_{i=0}^{n}a_{i}>d$.  We call $I$ the~index
of the~Fano variety $X$.
Note that $I$ should not be confused with the Fano index of $X$
(see Remark~\ref{remark:Bishop-Lichnerowicz}).

\begin{definition}
\label{definition:threshold} The~global log canonical threshold of
the~Fano variety $X$ is the~number
$$
\mathrm{lct}\big(X\big)=\mathrm{sup}\left\{\lambda\in\mathbb{Q}\ \left|\ %
\aligned
&\text{the~log pair}\ \Big(X, \lambda D\Big)\ \text{is log canonical}\\
&\text{for every effective $\mathbb{Q}$-divisor}\ D\equiv-K_{X}\\
\endaligned\right.\right\}\in\mathbb{R}.
$$
\end{definition}

The~number $\mathrm{lct}(X)$ is an~algebraic counterpart of
the~$\alpha$-invariant introduced in \cite{Ti87}. In particular,
the global log canonical threshold and the $\alpha$-invariant are known
to coincide in the nonosingular case (see e.\,g. \cite[Theorem~A.3]{ChSh08c}).
One of the important applications of (either of) these invariants
is the problem of existence of an orbifold K\"ahler--Einstein metric on
the variety~$X$.

\begin{theorem}[{\cite{Ti87}, \cite{DeKo01}}] \label{theorem:KE} The variety $X$ admits
an~orbifold K\"ahler--Einstein metric~if
$$
\mathrm{lct}\big(X\big)>\frac{\mathrm{dim}\big(X\big)}{\mathrm{dim}\big(X\big)+1}.
$$
\end{theorem}

There are Fano orbifolds that do not admit orbifold
K\"ahler--Einstein metrics (see \cite{Mat57},
\cite{Fu83},\cite{GaMaSpaYau06}).

\begin{theorem}[\cite{GaMaSpaYau06}]
\label{theorem:Bishop-Lichnerowicz} The variety $X$ admits no
K\"ahler--Einstein metrics if either $I>na_{0}$ or
$$
dI^{n}>n^n\prod_{i=0}^{n}a_{i}.%
$$
\end{theorem}

The two inequalities mentioned in Theorem~\ref{theorem:Bishop-Lichnerowicz}
are known as Lichnerowicz and Bishop obstructions, respectively.
A remarkable fact is that in our case they are not independent.
Namely, we prove the~following result in
Section~\ref{section:obstructions}.

\begin{theorem}
\label{theorem:Bishop-versus-Lichnerowicz} Let
$\bar{a}_{0}\leqslant\bar{a}_{1}\leqslant\ldots\leqslant\bar{a}_{n}$
and $\bar{d}$ be positive real numbers such that
$$
\bar{d}\Bigg(\sum_{i=0}^{n}\bar{a}_{i}-\bar{d}\Bigg)^{n}>n^n\prod_{i=0}^{n}\bar{a}_{i},
$$
and $\bar{d}<\sum_{i=0}^{n}\bar{a}_{i}$. Then
$\sum_{i=0}^{n}\bar{a}_{i}-\bar{d}>n\bar{a}_{0}$.
\end{theorem}

It is well-known that $I\leqslant n=\mathrm{dim}(X)+1$ if $X$ is
smooth. On the other hand, we know that
$$
dI^{n}>n^n\prod_{i=0}^{n}a_{i}\iff I\Big(-K_{X}\Big)^{n-1}>\Big(\mathrm{dim}\big(X\big)+1\Big)^{n}.%
$$

\begin{remark}
\label{remark:Bishop-Lichnerowicz} Let $U$ be a~smooth Fano
variety of dimension $m$. Define the Fano index $\gimel$ of $U$ to
be the maximal integer such that $-K_{U}\sim\gimel H$ for some
$H\in\mathrm{Pic}(U)$. Then the~inequality
$$
\gimel\Big(-K_{U}\Big)^{m}\leqslant\Big(\mathrm{dim}\big(U\big)+1\Big)^{m+1}
$$
fails in general if $m\gg 1$ (see \cite[Proposition~5.22]{Deb01}).
But we always have $\gimel\leqslant m+1$.
\end{remark}

Suppose that $n=3$. Then $X$ is a~del Pezzo surface with at most
quotient singularities, which is an interesting object of study,
in particular from the point of
the question of existence
of orbifold K\"ahler--Einstein metrics and Sasakian--Einstein
structures (see e.\,g.~\cite{JoKo01b},
\cite{Ara02}, \cite{BoGaNa02}, \cite{BoGaNa03}) and some others
(see e.\,g.~\cite{Elagin}).
The classification of
such surfaces $X$ with $I=1$ is known due to~\cite{JoKo01b}.

\begin{theorem}[{\cite[Theorem~8]{JoKo01b}}]
\label{theorem:Kollar-Johnson} Suppose that $I=1$. Then
\begin{itemize}
\item either $(a_{0},a_{1},a_{2},a_{3},d)=(2,2m+1,2m+1,4m+1,8m+4)$,
where $m$ is a positive integer,%
\item or the~quintuple $(a_{0},a_{1},a_{2},a_{3},d)$ lies in
the~sporadic set
$$
\left\{\aligned
&(1,1,1,1,3), (1,1,1,2,4), (1,1,2,3,6), (1,2,3,5,10),  \\%
&(1,3,5,7,15), (1,3,5,8,16), (2,3,5,9,18), (3,3,5,5,15), (3,5,7,11,25), \\%
&(3,5,7,14, 28), (3,5,11,18, 36), (5,14,17,21,56),(5,19,27,31,81),\\
&(5,19,27,50,100), (7,11,27,37,81), (7,11,27,44,88), (9,15,17,20,60),   \\
& (9,15,23,23,69), (11,29,39,49,127), (11,49,69,128,256), \\
& (13,23,35,57,127), (13,35,81,128,256) \\
\endaligned\right\}.
$$
\end{itemize}
\end{theorem}

Note that we can not apply Theorem~\ref{theorem:KE} to the~surface $X$ if
$I\geqslant 3a_{0}/2$, because $\mathrm{lct}(X)\leqslant a_{0}/I$.

The authors of~\cite{BoGaNa03} went further to classify the cases
with $2\leqslant I\leqslant 10$
and suggest that $I$ cannot attain larger values.

\begin{theorem}[{cf. \cite[Theorem~4.5]{BoGaNa03}}]
\label{theorem:BGN} Suppose that $2\leqslant I\leqslant 10$ and
$I<3a_{0}/2$. Then
\begin{itemize}
\item either
there exist a~non-negative integer $k<I$ and
a~positive integer $a\geqslant I+k$ such that
$$
\big(a_{0},a_{1},a_{2},a_{3},d\big)=\big(I-k,I+k,a,a+k,2a+k+I\big),
$$

\item or the~quintuple $(a_{0},a_{1},a_{2},a_{3},d)$ belongs to
one of the~following infinite series:
\begin{itemize}
\item $(3,3m,3m+1,3m+1, 9m+3)$,

\item $(3, 3m+1, 3m+2, 3m+2, 9m+6)$,%

\item $(3, 3m+1, 3m+2, 6m+1, 12m+5)$,

\item $(3,3m+1, 6m+1, 9m, 18m+3)$,%

\item $(3, 3m+1, 6m+1, 9m+3, 18m+6)$,%

\item $(4, 2m+3, 2m+3, 4m+4, 8m+12)$,%

\item $(4, 2m+3, 4m+6, 6m+7, 12m+18)$,%

\item $(6, 6m+3, 6m+5, 6m+5, 18m+15)$,%

\item $(6, 6m+5, 12m+8, 18m+9, 36m+24)$,%

\item $(6, 6m+5, 12m+8, 18m+15, 36m+30)$,%

\item $(8, 4m+5, 4m+7, 4m+9, 12m+23)$,%

\item $(9, 3m+8, 3m+11, 6m+13, 12m+35)$,%
\end{itemize}
where $m$ is a~positive integer,%

\item or the~quintuple $(a_{0},a_{1},a_{2},a_{3},d)$ lies in
the~sporadic set\footnote{We group these quintuples according to the value
of $I$, and in each group the quintuples are ordered lexicographically.}
$$
\left\{\aligned
&(2,3,4,7,14), (3,4,5,10,20), (3,4,6,7,18),
(3,4,10,15,30), (5,13,19,22,57),\\
&(5,13,19,35,70), (6,9,10,13,36), (7,8,19,25,57), (7,8,19,32,64),\\
&(9,12,13,16,48), (9,12,19,19,57), (9,19,24,31,81), (10,19,35,43,105),\\
&(11,21,28,47,105), (11,25,32,41,107), (11,25,34,43,111), (11,43,61,113,226),\\
&(13,18,45,61,135), (13,20,29,47,107), (13,20,31,49,111), (13,31,71,113,226),\\
&(14,17,29,41,99), (5,7,11,13,33), (5,7,11,20,40), (11,21,29,37,95),\\
&(11,37,53,98,196), (13,17,27,41,95), (13,27,61,98,196), (15,19,43,74,148),\\
&(9,11,12,17,45), (10,13,25,31,75), (11,17,20,27,71), (11,17,24,31,79), \\
&(11,31,45,83,166), (13,14,19,29,71), (13,14,23,33,79), (13,23,51,83,166), \\
&(11,13,19,25,63),(11,25,37,68,136), (13,19,41,68,136),(11,19,29,53,106),\\
&(13,15,31,53,106), (11,13,21,38,76), (3,7,8,13,29), (3,10,11,19,41),\\
&  (5,6,8,9,24), (5,6,8,15,30), (2,3,4,5,12), (7,10,15,19,45), \\
&  (7,18,27,37,81),(7,15,19,32,64), (7,19,25,41,82), (7,26,39,55,117). \\
\endaligned\right\}.
$$
\end{itemize}
\end{theorem}

Note that Theorem~~\ref{theorem:BGN}  differs from
\cite[Theorem~4.5]{BoGaNa03} in the~following way.
\begin{itemize}
\item The~series $(3, 3m+1, 3m+2, 6m+1, 12m+5)$
is omitted in \cite[Theorem~4.5]{BoGaNa03}.%

\item We have removed the~quintuple $(5,7,8,9,23)$ from the~list
of sporadic cases since
$(5,7,8,9,23)=(I-k,I+k,a,a+k,2a+k+I)$ for $I=6$, $k=1$ and $a=8$.%

\item The infinite series in
\cite[Theorem~4.5]{BoGaNa03} corresponding to our series
$(4, 2m+3, 4m+6, 6m+7, 12m+18)$ starts from $m=0$;
we have shifted it and extracted
the sporadic case
$(3,4,6,7,18)$ corresponding to $m=0$.

\item
The infinite series in
\cite[Theorem~4.5]{BoGaNa03} corresponding to our series
$(8, 4m+5, 4m+7, 4m+9, 12m+23)$ in
\cite[Theorem~4.5]{BoGaNa03} starts with $m=0$;
we have shifted it and extracted
the sporadic case
$(5,7,8,9,23)$ corresponding to $m=0$.

\item The infinite series in
\cite[Theorem~4.5]{BoGaNa03} corresponding to our series
$(9, 3m+8, 3m+11, 6m+13, 12m+35)$ starts with $m=-1$; we have shifted it and extracted
the sporadic case $(8,9,11,13,35)$ corresponding to $m=0$ (note
that the quintuple $(5,7,8,9,23)$ corresponding to $m=-1$ has already appeared
from the previous series).%
\end{itemize}

\begin{remark}
\label{remark:BGN-small} Arguing as in the~proof of
\cite[Lemma~5.2]{BoGaNa03}, one~can~show~that
$$
\mathrm{lct}\big(X\big)\geqslant 2/3\iff\big(a_{0},a_{1},a_{2},a_{3},d\big)\in\Big\{ \big(1,1,1,1,3\big), \big(1,1,2,3,6\big)\Big\}%
$$
in the case when
$(a_{0},a_{1},a_{2},a_{3},d)=(I-k,I+k,a,a+k,2a+k+I)$ for some
non-negative integer $k<I$ and some positive integer $a\geqslant
I+k$ (cf. \cite[Theorem~1.7]{Ch07b}). These two cases are exactly
ones when $X$ is smooth.
\end{remark}

The main purpose of this paper is to prove a technical result (see
Theorem~\ref{theorem:technical} in
Section~\ref{section:examples}), which we derive from the classification
of isolated quasi-homogeneous rational three-dimensional
hypersurface
singularities obtained in~\cite{YauYu03}. Being not very attractive on its own,
Theorem~\ref{theorem:technical} easily implies
the~following.

\begin{theorem}
\label{theorem:Cheltsov-Shramov-BGN} The assertion of
Theorem~\ref{theorem:BGN} holds without the~assumption $I\leqslant
10$.
\end{theorem}

Therefore, we obtain a proof of the (corrected version of the)
half-experimental result of~\cite{BoGaNa03} (i.\,e. Theorem~\ref{theorem:BGN})
modulo~\cite{YauYu03}.

As the second application of
Theorem~\ref{theorem:technical} we derive from it a
classificational result in the style of~\cite{JoKo01b} which
is more explicit than the corresponding result of~\cite{BoGaNa03}.
Namely, we list the cases with $I=2$. Note that obtaining the list of
the cases with any bounded index requires just  a bit  of elementary
computation modulo Theorem~\ref{theorem:technical}.

\begin{corollary}
\label{theorem:Cheltsov-Shramov-I-2} Suppose that $I=2$. Then
\begin{itemize}

\item either $(a_{0},a_{1},a_{2},a_{3},d)=(1,1,s,r,s+r)$, where
$s\le r$ are positive integers,

\item or the~quintuple $(a_{0},a_{1},a_{2},a_{3},d)$ belongs to
one of the~following infinite series:
\begin{itemize}
\item $(1, 2, m+1, m+2, 2m+4)$,%
\item $(1, 3, 3m, 3m+1, 6m+3)$,%
\item $(1, 3, 3m+1, 3m+2, 6m+5)$,%
\item $(3, 3m, 3m+1, 3m+1, 9m+3)$,%
\item $(3, 3m+1, 3m+2, 3m+2, 9m+6)$,%
\item $(3, 3m+4, 3m+5, 6m+7, 12m+17)$,
\item $(3, 3m+1, 6m+1, 9m, 18m+3)$,%
\item $(3, 3m+1, 6m+1, 9m+3, 18m+6)$,%
\item $(4, 2m+3, 2m+3, 4m+4, 8m+12)$,%
\item $(4, 2m+3, 4m+6, 6m+7, 12m+18)$,%
\end{itemize}
where $m$ is a~positive integer,%

\item or the~quintuple $(a_{0},a_{1},a_{2},a_{3},d)$ lies in
the~sporadic set
$$
\left\{\aligned
&(1,1,2,2,4),(1,4,5,7,15),(1,4,5,8,16),(1,5,7,11,22),(1,6,9,13,27),\\%
(1,7,12,18,36),
&(1,8,13,20,40),(1,9,15,22,45),(1,3,4,6,12),(1,4,6,9,18),\\%
(1,6,10,15,30),
&(2,3,4,5,12), (2,3,4,7,14), (3,4,5,10,20),(3,4,6,7,18), (3,4,10,15,30),\\%
(3,4,6,7,18),
&  (5,13,19,22,57),(5,13,19,35,70), (6,9,10,13,36), (7,8,19,25,57),\\
& (7,8,19,32,64),(9,12,13,16,48), (9,12,19,19,57), (9,19,24,31,81),\\
& (10,19,35,43,105),(11,21,28,47,105), (11,25,32,41,107), (11,25,34,43,111),\\
& (11,43,61,113,226),(13,18,45,61,135), (13,20,29,47,107),\\
& (13,20,31,49,111), (13,31,71,113,226),(14,17,29,41,99)\\
\endaligned\right\}.
$$
\end{itemize}
\end{corollary}

As was already mentioned above, an interesting question about a surface
$X$ is whether $X$ admits an orbifold K\"ahler--Einstein metric or not.
Some obstructions are provided by Theorem~\ref{theorem:Bishop-Lichnerowicz},
and the main instrument to prove the existence is the sufficient
condition given by Theorem~\ref{theorem:KE}. Most of the examples
mentioned in Theorems~\ref{theorem:Kollar-Johnson} and~\ref{theorem:BGN}
have already been studied from this point of view.
As for the series omitted in~\cite{BoGaNa03}, we have the following.

\begin{theorem}
\label{theorem:Cheltsov-Shramov-new-series} Suppose that
$$
\big(a_{0},a_{1},a_{2},a_{3},d\big)=\big(3,3m+1,3m+2,6m+1,12m+5\big),%
$$
where $m\in\mathbb{Z}_{>0}$. Then $\mathrm{lct}(X)=1$.
\end{theorem}

Theorem~\ref{theorem:Cheltsov-Shramov-new-series} can be proved along the
same lines as the results of~\cite{ChShPa08}.

\smallskip
The~results of~\cite{Ti90}, \cite{JoKo01b}, \cite{Ara02},
\cite{BoGaNa02}, \cite{BoGaNa03}, \cite{ChShPa08} together with
Theorem~\ref{theorem:Cheltsov-Shramov-new-series}
imply the~following result concerning orbifold K\"ahler--Einstein
metrics on the Del Pezzo hypersurfaces $X$.

\begin{corollary}
\label{corollary:BGN-CPS} Suppose that $I<3a_0/2$. Then
\begin{itemize}
\item either $X$ admits an~orbifold K\"ahler--Einstein metric,%
\item or one of the~following possible exceptions occur:
\begin{itemize}
\item there exist a~non-negative integer $k<I$ and a~positive
integer $a\geqslant I+k$ such that
$$
\big(a_{0},a_{1},a_{2},a_{3},d\big)=\big(I-k,I+k,a,a+k,2a+k+I\big),
$$

\item the~quintuple $(a_{0},a_{1},a_{2},a_{3},d)$ lies in the~set
$$
\left\{\aligned
&(2,3,4,7,14), (7,10,15,19,45), (7,18,27,37,81),\\
&(7,15,19,32,64), (7,19,25,41,82), (7,26,39,55,117) \\
\endaligned\right\},
$$
\item $(a_{0},a_{1},a_{2},a_{3},d)=(1,3,5,7,15)$ and $\phi(x_{0},x_{1},x_{2},x_{3})$ does not contain $x_{1}x_{2}x_{3}$,%
\item $(a_{0},a_{1},a_{2},a_{3},d)=(2,3,4,5,12)$ and $\phi(x_{0},x_{1},x_{2},x_{3})$ does not contain $x_{1}x_{2}x_{3}$.%
\end{itemize}
\end{itemize}
\end{corollary}

\begin{remark}\label{remark:mnogo}
One can show that there are infinitely many quintuples
$$(i-k,i+k,a,a+k,2a+k+i)$$
such~that there exists a~quasi\-smooth
well-formed hypersurface in $\mathbb{P}(i-k,i+k,a,a+k)$ of
degree~$2a+k+i$, where $k$, $a$, $i$ are non-negative integers
such that $0\leqslant k<i$ and $a\geqslant i+k$.
\end{remark}

\begin{example}
\label{example:BGN}
A~general hypersurface in $\mathbb{P}(1, 2n-1,
2n-1, 3n-2)$ of degree~$6n-3$ is~a~quasi\-smooth well-formed del
Pezzo surface for every positive integer $n$.
This series corresponds to the values $k=n-1$, $a=2n-1$ and $i=n$
of Remark~\ref{remark:mnogo}.
\end{example}

We thank C.\,Boyer, B.\,Nill, D.\,Orlov, J.\,Park, J.\,Stevens,
G.\,Tian, S.\,S.-T.\,Yau and Y.\,Yu for very useful discussions.
Special thanks go to Laura Morris for checking the
computations in~\cite{YauYu03} and to
Erik Paemurru for finding a gap in an earlier
version of our paper.

We are grateful to Pohang Mathematics Institute (PMI) for
hospitality.

\section{Technical result}
\label{section:examples}

Let $X$ be a~quasismooth hypersurface in
$\mathbb{P}(a_{0},a_{1},a_{2},a_{3})$ of degree $d$ (throughout this section
we will not assume that the numbers $a_i$ are ordered).
The hypersurface $X$ is
given~by
$$
\phi\big(x,y,z,w\big)=0\subset\mathbb{P}\big(a_{0},a_{1},a_{2},a_{3}\big)\cong\mathrm{Proj}\Big(\mathbb{C}\big[x,y,z,w\big]\Big),
$$
where $\mathrm{wt}(x)=a_{0}$, $\mathrm{wt}(y)=a_{1}$,
$\mathrm{wt}(z)=a_{2}$, $\mathrm{wt}(w)=a_{3}$,
and $\phi(x,y,z,w)$ is a~quasihomogeneous polynomial of degree $d$.

\begin{definition}
\label{definition:degenerate}
We say that $X$ is degenerate if $d=a_{i}$
for some $i$ (cf. \cite[Definition~6.5]{IF00}).%
\end{definition}

The purpose of this section is to prove the following result.

\begin{theorem}
\label{theorem:technical} Suppose that $a_0\le\ldots\le a_3$, and the
hypersurface
$X\subset\mathbb{P}(a_{0},a_{1},a_{2},a_{3})$
is a well-formed non-degenerate del Pezzo surface. Then
\begin{itemize}
\item either  there exist a~non-negative integer $k<I$ and
a~positive integer $a\geqslant I+k$ such that
$$
\big(a_{0},a_{1},a_{2},a_{3},d\big)=\big(I-k,I+k,a,a+k,2a+k+I\big),
$$
\item or $I=a_{i}+a_{j}$ for some distinct $i$ and $j$,
\item or $I=a_{i}+\frac{a_{j}}{2}$ for some distinct $i$ and $j$,%
\item or $(a_{0},a_{1},a_{2},a_{3},d,I)$ belongs to one of
the~infinite series listed in
Table~\ref{table:1},%
\item or $(a_{0},a_{1},a_{2},a_{3},d,I)$ lies in the~sporadic set listed in Table~\ref{table:2}.%
\end{itemize}
\end{theorem}

\begin{remark}
Note that the first three cases of Theorem~\ref{theorem:technical}
are not mutually exclusive. On the other hand,
since the most interesting cases (say, from the point
of view of K\"ahler--Einstein metrics) appear in the last two cases
of Theorem~\ref{theorem:technical}, we designed the tables
in Appendix~\ref{section:infinite-series} so that the cases listed
there are mutually exclusive, and none of them is contained in
any of the first three cases of Theorem~\ref{theorem:technical}.
One can check that for each sixtuple
$(a_{0},a_{1},a_{2},a_{3},d,I)$ listed in Table~\ref{table:1} and
\ref{table:2}, there exists a~well-formed quasismooth hypersurface
in $\mathbb{P}(a_{0},a_{1},a_{2},a_{3})$ of degree~$d$
(apparently, this is not the case with the first three cases of
Theorem~\ref{theorem:technical}).
\end{remark}

\begin{remark}
\label{remark:lct-special-cases} If $I=a_{i}+a_{j}$ or
$I=a_{i}+a_{j}/2$ for some $i$ and $j$, then
$\mathrm{lct}(X)\leqslant 2/3$.
Unfortunately, we do not know
how to handle the problem of existence of K\"ahler--Einstein
metrics in these cases. Neither we know this for
the first case of Theorem~\ref{theorem:technical}.
Note that the Bishop and Lichnerowicz
obstructions (see Theorem~\ref{theorem:Bishop-Lichnerowicz})
are not enough to settle this question.
\end{remark}

\bigskip
The proof of Theorem~\ref{theorem:technical} is based on the
classification of isolated three-dimensional
quasihomogeneous rational hypersurface singularities.
Consider a singularity $(V,O)$
defined by the~equation
$$
\phi\big(x,y,z,w\big)=0\subset\mathbb{C}^{4}\cong\mathrm{Spec}\Big(\mathbb{C}\big[x,y,z,w\big]\Big),
$$
where
$O$ is the~origin of $\mathbb{C}^{4}$.
Suppose that $(V,O)$ is an isolated singularity (this happens if and
only if the corresponding hypersurface
$X\subset\mathbb{P}(a_{0},a_{1},a_{2},a_{3})$ is quasismooth).
Suppose also that $V$ is indeed
singular at the~point $O$, i.\,e.
$\mathrm{mult}_{O}(V)\ge 2$ (this happens if and
only if the corresponding hypersurface
$X\subset\mathbb{P}(a_{0},a_{1},a_{2},a_{3})$ is non-degenerate).
The following classificational result may be obtained by studying
Newton diagrams of the corresponding polynomials.

\begin{theorem}[{\cite[Theorem~2.1]{YauYu03}}]
\label{theorem:Newton}
One has
$$
\phi\big(x,y,z,w\big)=\xi\big(x,y,z,w\big)+\chi\big(x,y,z,w\big)
$$
where $\xi(x,y,z,w)$ and $\chi(x,y,z,w)$ are~quasihomogeneous
polynomials of degree $d$ with respect to the~weights
$\mathrm{wt}(x)=a_{0}$, $\mathrm{wt}(y)=a_{1}$,
$\mathrm{wt}(z)=a_{2}$, $\mathrm{wt}(t)=a_{3}$ such that
the~quasihomogeneous polynomials $\xi(x,y,z,w)$ and
$\chi(x,y,z,w)$ do not have common monomials,~the~equation
$$
\xi\big(x,y,z,w\big)=0\subset\mathbb{C}^{4}\cong\mathrm{Spec}\Big(\mathbb{C}\big[x,y,z,w\big]\Big),
$$
defines an~isolated  singularity, and $\xi(x,y,z,w)$ is one of
the~following polynomials:
\begin{itemize}
\item[\textbf{I}] $Ax^{\alpha}+By^{\beta}+Cz^{\gamma}+Dw^{\delta}$,%
\item[\textbf{II}] $Ax^{\alpha}+By^{\beta}+Cz^{\gamma}+Dzw^{\delta}$,%
\item[\textbf{III}] $Ax^{\alpha}+By^{\beta}+Cz^{\gamma}w+Dzw^{\delta}$,%
\item[\textbf{IV}] $Ax^{\alpha}+Bxy^{\beta}+Cz^{\gamma}+Dzw^d$,%
\item[\textbf{V}] $Ax^{\alpha}y+Bxy^{\beta}+Cz^{\gamma}+Dzw^{\delta}$,%
\item[\textbf{VI}] $Ax^{\alpha}y+Bxy^{\beta}+Cz^{\gamma}w+Dzw^{\delta}$,%
\item[\textbf{VII}] $Ax^{\alpha}+By^{\beta}+Cyz^{\gamma}+Dzw^{\delta}$,%
\item[\textbf{VIII}] $Ax^{\alpha}+By^{\beta}+Cyz^{\gamma}+Dyw^{\delta}+Ez^{\epsilon}w^{\zeta}$,%
\item[\textbf{IX}] $Ax^{\alpha}+By^{\beta}w+Cz^{\gamma}w+Dyw^{\delta}+Ey^{\epsilon}z^{\zeta}$,%
\item[\textbf{X}] $Ax^{\alpha}+By^{\beta}z+Cz^{\gamma}w+Dyw^{\delta}$,%
\item[\textbf{XI}] $Ax^{\alpha}+Bxy^{\beta}+Cyz^{\gamma}+Dzw^{\delta}$,%
\item[\textbf{XII}] $Ax^{\alpha}+Bxy^{\beta}+Cxz^{\gamma}+Dyw^{\delta}+Ey^{\epsilon}z^{\zeta}$,%
\item[\textbf{XIII}] $Ax^{\alpha}+Bxy^{\beta}+Cyz^{\gamma}+Dyw^{\delta}+Ez^{\epsilon}w^{\zeta}$,%
\item[\textbf{XIV}] $Ax^{\alpha}+Bxy^{\beta}+Cxz^{\gamma}+Dxw^{\delta}+Ey^{\epsilon}z^{\zeta}+Fz^{\eta}w^{\theta}$,%
\item[\textbf{XV}] $Ax^{\alpha}y+Bxy^{\beta}+Cxz^{\gamma}+Dzw^{\delta}+Ey^{\epsilon}z^{\zeta}$,%
\item[\textbf{XVI}] $Ax^{\alpha}y+Bxy^{\beta}+Cxz^{\gamma}+Dxw^{\delta}+Ey^{\epsilon}z^{\zeta}+Fz^{\eta}w^{\theta}$,%
\item[\textbf{XVII}] $Ax^{\alpha}y+Bxy^{\beta}+Cyz^{\gamma}+Dxw^{\delta}+Ey^{\epsilon}w^{\zeta}+Fx^{\eta}z^{\theta}$,%
\item[\textbf{XVIII}] $Ax^{\alpha}z+Bxy^{\beta}+Cyz^{\gamma}+Dyw^{\delta}+Ez^{\epsilon}w^{\zeta}$,%
\item[\textbf{XIX}] $Ax^{\alpha}z+Bxy^{\beta}+Cz^{\gamma}w+Dyw^{\delta}$,%
\end{itemize}
where $\alpha,\beta,\gamma,\delta$ are positive integers,
$\epsilon,\zeta,\eta,\theta$ are non-negative integers, and
$A,B,C,D,E,F$ are complex numbers.
\end{theorem}

We will refer to the latter polynomials according to case labelling
in Theorem~\ref{theorem:Newton}.
For simplicity of notations, we suppose that $A=B=C=D=E=F=1$
in the rest of the paper\footnote{The singularity defined by
$\xi(x, y, z, w)$ is not necessary isolated if $A=\ldots=F=1$
(this happens, for instance, in the case XIX if
$\alpha=\beta=\gamma=\delta=1$). We hope that such abuse of
notation will not lead to a confusion.}.

\bigskip

In order to prove Theorem~\ref{theorem:technical} we will
suppose that $d<\sum_{i=0}^{3}a_{i}$ (this happens if and
only if~$X$ is a del Pezzo surface, provided that~$X$ is well-formed).
Then the singularity $(V,O)$ is canonical
(see Remark~\ref{remark:rational-vs-canonical-etc}), and
thus $\mathrm{mult}_{O}(V)\le 3$.
Moreover, the~singularity $(V,O)$ is rational
(see Remark~\ref{remark:rational-vs-canonical-etc}).
The main result of~\cite{YauYu03} is a classification of
(the deformation families of) the quasihomogeneous polynomials
that define isolated three-dimensional
quasihomogeneous \emph{rational} hypersurface singularities
up to an analytical change of coordinates (in some sense it is
a refinement of Theorem~\ref{theorem:Newton}).
To give a classification
of quasismooth del Pezzo hypersurfaces in the
weighted projective spaces we actually need the classification
of such polynomials up to
the change of coordinates that is compatible with the corresponding
$\C^*$-action (i.\,e., the change of coordinates that respects the
weights).\footnote{
Note that these two classifications indeed differ. Say,
if one denotes by
$\upsilon(x,y,z,w)$
the~$(\alpha,\beta,\gamma,\delta)$-part of the~polynomial
$\xi(x,y,z,w)$, one sees that the cases when
$\upsilon$ has less than $4$ different monomials
are absent from the list of~\cite{YauYu03}.
These are
$\xi(x,y,z,w)=x^{\alpha}+y^{\beta}+z^{\gamma}w+zw^{\delta}$
with $\gamma=\delta=1$ (cf.~\cite[Case~III]{YauYu03}),
$\xi(x,y,z,w)=x^{\alpha}y+xy^{\beta}+z^{\gamma}+zw^{\delta}$ with
$\alpha=\beta=1$ (cf.~\cite[Case~V]{YauYu03}),
and
$\xi(x,y,z,w)=x^{\alpha}y+xy^{\beta}+z^{\gamma}w+zw^{\delta}$ with
$\alpha=\beta=1$ or/and $\gamma=\delta=1$
(cf.~\cite[Case~VI]{YauYu03}).
It is easy to check that the listed cases are
equivalent up to an analytical change of coordinates to
some other cases that are present in the list of~\cite{YauYu03},
but one can choose the weights of variables so that there does not
exist such change of coordinates that respects the weights.}
Indeed, while the weights of the variables are not fixed even if one fixes
a polynomial~$\xi(x,y,z,w)$ from Theorem~\ref{theorem:Newton}
that is homogeneous with respect to these weights
(since one can multiply all of them by some constant), the corresponding
\emph{well-formed} weighted projective space and thus
the family of the corresponding well-formed hypersurfaces
becomes fixed in this case.
Fortunately, these two classifications are not very far from each other.
To recover the latter from the former is not a difficult task,
but still it requires some
additional work. Luckily, to prove Theorem~\ref{theorem:technical}
we don't need to do it in full generality,
since we can disregard polynomials whose degree~$d$
(and thus the index $I$ either)
equals a sum of two of the weights. The latter
are included in one of the types of our resulting classification
(see Theorem~\ref{theorem:technical}).
If there is a \emph{unique} choice of weights
$\wt(x)$, $\wt(y)$, $\wt(z)$, $\wt(w)$ that makes some of the polynomials
obtained from
the polynomial~$\xi(x,y,z,w)$ by an analytical change of coordinates
quasihomogeneous, then one trivially obtains that any change of coordinates
that turns~$\xi$ into another quasihomogeneous polynomial
must agree with the corresponding $\C^*$-action.
Furthermore, this is the case if we restrict ourselves to the weights
that are at most~$d/2$, where $d$ is the total weight of a corresponding
polynomial (see~\cite[Lemma~4.3]{Saito}).
Therefore, the polynomials that we need to recover must be homogeneous with
respect to
the weights such that one of the weights, say~$\wt(x)$, is strictly larger
than~$d/2$.
In this case we have $\xi=xg+h$, where~$g$ and~$h$ are polynomials that do not
depend on~$x$. By quasismoothness at least one other variable
occurs linearly in $g$, so by a $\C^*$-equivariant coordinate transformation
we may assume that $g$ is a coordinate, say $y$. Now collect all terms
divisible by $y$ and absorb them in $xy$ by a ($\C^*$-equivariant)
coordinate change in $x$. Still we have to take care of all polynomials
that are obtained from~$\xi$ by an analytical change of coordinates
(note that these may not contain a monomial that is a product of two variables
even if~$\xi$ does).
The rank of the hypersurface singularity in question
is at least $2$.
The latter is preserved under the analytical
change of coordinates, so it is enough for our purposes
to describe all possible quasihomogeneous polynomials $f$ (say,
in variables $x_0$, $x_1$, $x_2$ and $x_4$)
giving a singularity of rank $r$ equal to $2$, $3$ or $4$,
and not containing monomials $x_ix_j$ for $i\neq j$.
The latter condition implies that (up to $\C^*$-equivariant
coordinate change) $f=x_0^2+\ldots+ x_r^2+
g(x_{r+1}, \ldots, x_4)$, where $g$ is a polynomial in $4-r$ variables
of rank~$0$ (i.\,e. corank $0\le 4-r\le 2$). If $r=4$, then $g=0$,
and if $r=3$, then $g=x_3^n$, so that in both of these cases
$f$ is found in~\cite[Case~I]{YauYu03}. If $r=2$,
applying~\cite[\S13.1]{AGV} (and keeping in mind~\cite[Lemma~4.3]{Saito}),
we again see that
$f$ is contained in the list of~\cite{YauYu03}
(cases~I.1, II.1 and~III.1).\footnote{We are grateful
to J.\,Stevens who explained this argument to us.}

To summarize,
for every~$\xi(x,y,z,w)$ the~possible
values (up to a $\C^*$-equivariant change of coordinates)
of the~quadruple
$(\alpha,\beta,\gamma,\delta)$ are listed
in~\cite{YauYu03}
up to the polynomials that contain a monomial which is a product of two
variables.
Unfortunately, as it usually happens with long lists,
in the list of~\cite{YauYu03} there are some omissions.
Namely, apart from minor misprints (see
Examples~\ref{example:2} and~\ref{example:8} below)
the following cases are omitted\footnote{
We are grateful to L.\,Morris who checked the
computations of~\cite{YauYu03} and found these omissions.}
\begin{itemize}
\item [\textbf{XI}] $\xi(x,y,z,w)=x^{\alpha}+xy^{\beta}+yz^{\gamma}+zw^{\delta}$ and $(\alpha,\beta,\gamma,\delta)=(2,4,13,3)$,%
\item [\textbf{XII}]
$\xi(x,y,z,w)=x^{\alpha}+xy^{\beta}+xz^{\gamma}+yw^{\delta}+y^{\epsilon}z^{\zeta}$
and
$$
\big(\alpha,\beta,\gamma,\delta,\epsilon,\zeta\big)\in\Big
\{\big(5,4,3,2,1,3\big),\big(7,4,3,2,2,2\big),\big(6,5,3,2,1,3\big)\Big\}.
$$
\end{itemize}

\begin{remark}\label{remark:twice}
Note that different cases in the list of~\cite{YauYu03} are
not mutually exclusive. For example, for~\cite[Case~I.1]{YauYu03}
with $r=s=2$
and~\cite[Case~XIII.1(7)]{YauYu03} with $r=2$ there is a $\C^*$-action
and a change of coordinates equivariant with respect to this
action such that the two (deformation families of) the singularities
are the same (actually, such coincidences are numerous in~\cite{YauYu03}).
A side effect of this is that sometimes one has to make a
($\C^*$-equivariant) coordinate change to find a given polynomial in the
list of~\cite{YauYu03}. For example,
the polynomial $\xi=x^4+xy^4+xz^3+yw^2+y^4z$ is not found
in~\cite[Case~XII]{YauYu03} as one could possibly expect,
but in the new coordinates $x'=x$, $y'=z-x$, $z'=y$ and $w'=w$ it gives
the same deformation family as~\cite[Case~XI.3(16)]{YauYu03}
for $r=s=4$.
\end{remark}

\bigskip
Therefore, given a list of~\cite{YauYu03},
to prove Theorem~\ref{theorem:technical}, we must find
all singularities in this list that correspond to the well-formed
hypersurfaces $X\subset\P(a_0,a_1,a_2,a_3)$. This means that we need to find
all~possible values of the~quadruple
$(\alpha,\beta,\gamma,\delta)$ such that
$$
\mathrm{gcd}\big(a_i,a_{j},a_k\big)=1
$$
and $d$ is divisible by $\mathrm{gcd}(a_i,a_{j})$ for all $i\ne
j\ne k\ne i$. Let us show how to do this in the few typical cases.

\begin{example}
\label{example:1} Suppose that the~hypersurface $X$ is
well-formed, and suppose that the~quasihomogeneous polynomial
$\xi(x,y,z,w)$ is found in the~third part
of~\cite[Case~X.3(1)]{YauYu03}. Then
$$
\xi\big(x,y,z,w\big)=x^2+y^3z+z^5w+yw^u,
$$
where $5\leqslant u\leqslant 18$. Hence
$2a_0=3a_1+a_2=5a_2+a_3=a_1+ua_3$. Put $a_3=a$. Then
$$
\big(a_0, a_1, a_2, a_3,d\big)=\left(\frac{(15u+1)a}{22}, \frac{(4u+1)a}{11}, \frac{(3u-2)a}{11}, a, \frac{(15u+1)a}{11}\right),%
$$
where either $a=1$ or $a=11$, because $\gcd(a_1, a_2, a_3)=1$.

Suppose that $a=1$. Then $3u-2$ and $4u+1$ are divisible by $11$.
We see that $u=8$. Then
$$
a_0=\frac{(15u+1)a}{22}=\frac{121}{22}\not\in\mathbb{Z},
$$
which is a~contradiction.

We see that $a=11$. Then $u$ must be odd for $a_0$ to be integer.
Thus, we obtain $7$ solutions:
\begin{itemize}
\item $(a_0, a_1, a_2, a_3,d,I)=(38, 21, 13, 11, 76, 7)$,%
\item $(a_0, a_1, a_2, a_3,d,I)=(53, 29, 19, 11, 106, 6)$,%
\item $(a_0, a_1, a_2, a_3,d,I)=(68, 37, 25, 11, 136, 5)$,%
\item $(a_0, a_1, a_2, a_3,d,I)=(83, 45, 31, 11, 166, 4)$,%
\item $(a_0, a_1, a_2, a_3,d,I)=(98, 53, 37, 11, 196, 3)$,%
\item $(a_0, a_1, a_2, a_3,d,I)=(113, 61, 43, 11, 226, 2)$,%
\item $(a_0, a_1, a_2, a_3,d,I)=(128, 69, 49, 11, 256, 1)$.
\end{itemize}
\end{example}

\begin{example}
\label{example:2} Suppose that the~hypersurface $X$ is
well-formed, and suppose that the~quasihomogeneous polynomial
$\xi(x,y,z,w)$ is found in the~second part
of~\cite[Case~XII.3(16)]{YauYu03}. Then\footnote{Note that there
is a~misprint in~\cite[Case~XII.3(16)]{YauYu03}, and one should
read $(5, 4)$ instead of $(4, 5)$.}
$$
\xi\big(x,y,z,w\big)= x^3+xy^5+xz^2+yw^4+y^{\epsilon}z^{\zeta},
$$
which gives $3a_0=a_0+5a_1=a_0+2a_2=a_1+4a_3$, which contradicts
the~well-formedness of $X$.
\end{example}

\begin{example}
\label{example:3}  Suppose that the~hypersurface $X$ is
well-formed, and suppose that the~quasihomogeneous polynomial
$\xi(x,y,z,w)$ is found in \cite[Case~I.2]{YauYu03}. Then
$$
\xi\big(x,y,z,w\big)=x^2+y^3+z^3+w^r,
$$
where $r\in\mathbb{Z}_{\geqslant 3}$. Hence $2a_0=3a_1=3a_2=ra_3$.
Thus $a_0=3$ and $a_1=a_2=2$, because
$$
\gcd\big(a_0, a_1, a_2\big)=1.%
$$

We see that $a_3=6/r$. Since $r\geqslant 3$, we have $a_3=1$,
because $\gcd(a_1, a_2, a_3)=1$.
\end{example}

\begin{example}
\label{example:4} Suppose that the~hypersurface $X$ is
well-formed, and suppose that the~quasihomogeneous polynomial
$\xi(x,y,z,w)$ is found in the~fourth part
of~\cite[Case~IX.3(3)]{YauYu03}. Then
$$
\xi\big(x,y,z,w\big)=x^3+y^2w+z^6w+yw^s+y^{\epsilon}z^{\zeta},
$$
where $s\in\mathbb{Z}_{\geqslant 6}$. Hence
$3a_0=2a_1+a_3=6a_2+a_3=a_1+sa_3$. Put $a_3=a$. Then
$$
\big(a_0, a_1, a_2, a_3,d\big)=\left(\frac{(2s-1)a}{3}, (s-1)a, \frac{(s-1)a}{3}, a, (2s-1)a\right),%
$$
where $d$ is divisible by $\gcd(a_1, a_2)=(s-1)a/3$. Thus, we have
$$
s-1\mid 3\big(2s-1\big),
$$
which is possible only if $3$ is divisible by $s-1$, which
contradicts the~assumption $s\geqslant 6$.
\end{example}

\begin{example}
\label{example:5} Suppose that the~hypersurface $X$ is
well-formed, and suppose that the~quasihomogeneous polynomial
$\xi(x,y,z,w)$ is found in the~first part
of~\cite[Case~VIII.3(5)]{YauYu03}. Then
$$
\xi\big(x,y,z,w\big)=x^2+y^s+yz^3+yw^3+z^{\epsilon}w^{\zeta},
$$
where $s\in\mathbb{Z}_{\geqslant 4}$. Hence
$2a_0=sa_1=a_1+3a_2=a_1+3a_3$. Put $a_1=a$. Then
$$
\big(a_0, a_1, a_2, a_3,d\big)=\left(\frac{sa}{2}, a, \frac{(s-1)a}{3}, \frac{(s-1)a}{3}, sa\right),%
$$
where $d=sa$ is divisible by $\gcd(a_2, a_3)=(s-1)a/3$, because
$X$ is well-formed. Thus
$$
s-1\mid 3s,
$$
which implies that $s=4$, because $s\geqslant 4$. Hence, we have
$$
\big(a_0, a_1, a_2, a_3,d\big)=\big(2a, a, a, a, 4a\big),
$$
which gives $a=1$. Then $X$ is a~smooth del Pezzo surface $X$ such
that $K_X^2=2$.
\end{example}

\begin{example}
\label{example:6} Suppose that the~hypersurface $X$ is
well-formed, and suppose that the~quasihomogeneous polynomial
$\xi(x,y,z,w)$ is found in the~second part
of~\cite[Case~XVIII.2(2)]{YauYu03}. Then
$$
\xi\big(x,y,z,w\big)=x^2z+xy^2+yz^s+yw^3+z^{\epsilon}w^{\zeta},
$$
where $s\in\mathbb{Z}_{\geqslant 4}$. Hence
$2a_0+a_2=a_0+2a_1=a_1+sa_2=a_1+3a_3$. Put $a_2=a$. Then
$$
\big(a_0, a_1, a_2, a_3,d\big)=\left(\frac{(2s-1)a}{3}, \frac{(s+1)a}{3}, a, \frac{sa}{3}, \frac{(4s+1)a}{3}\right).%
$$

Since either $s$ or $s+1$ is not divisible by $3$, we see that $a$
is divisible by $3$. But
$$
\gcd\big(a_0, a_1, a_2\big)=1,
$$
because $X$ is well-formed. Then $a=3$. Thus, we have
$$
\big(a_0, a_1, a_2, a_3,d\big)=\big(2s-1, s+1, 3, s, 4s+1\big),
$$
where $s\in\mathbb{Z}_{\geqslant 4}$. Note that if $s=2 \mod 3$,
then
$$
\gcd\big(a_0, a_1, a_2\big)=3,
$$
which is impossible. Then either $s=0 \mod 3$ or $s=1 \mod 3$.

Suppose that $s=0 \mod 3$. Then $s=3n$ for some
$n\in\mathbb{Z}_{\geqslant 2}$. We have
$$
\big(a_0, a_1, a_2, a_3,d\big)=\big(6n-1, 3n+1, 3, 3n, 12n+1\big),%
$$
and $d$ is not divisible by $\gcd(a_2, a_3)=3$, which contradicts
 the~well-formedness of $X$.

We see that $s=1 \mod 3$. Then $s=3n+1$ for some
$n\in\mathbb{Z}_{\geqslant 2}$. We have
$$
\big(a_0, a_1, a_2, a_3,d,I\big)=\big(6n+1, 3n+2, 3, 3n+1, 12n+5, 2\big).%
$$
\end{example}

\begin{example}
\label{example:7} Suppose that the~hypersurface $X$ is
well-formed, and suppose that the~quasihomogeneous polynomial
$\xi(x,y,z,w)$ is found in~\cite[Case~IX.2(1)]{YauYu03}. Then
$$
\xi\big(x,y,z,w\big)=x^2+y^2w+z^rw+yw^s+y^{\epsilon}z^{\zeta},
$$
where $r\in\mathbb{Z}_{\geqslant 2}\ni s$. Hence
$2a_0=2a_1+a_3=ra_2+a_3=a_1+sa_3$. Put $a_2=a$. Then
$$
\big(a_0, a_1, a_2, a_3,d\big)=\left(\frac{(2s-1)ra}{4(s-1)}, \frac{ra}{2}, a, \frac{ra}{2(s-1)}, \frac{(2s-1)ra}{2(s-1)}\right).%
$$

Note that $\gcd(2s-1,4(s-1))=1$. Thus $ra$ is divisible by
$4(s-1)$. But
$$
\gcd\big(a_0, a_1, a_3\big)=1,
$$
because the~hypersurface $X$ is well-formed. Then $ra=4(s-1)$.
Hence, we have
$$
\big(a_0, a_1, a_2, a_3,d\big)=\left(2s-1, 2s-2, \frac{4s-4}{r}, 2, 4s-2\right),%
$$
where $d$ is divisible by $\gcd(a_1, a_2)$. Hence $r(4s-2)$ is
divisible by $s-1$. Then
$$
r=k(s-1)
$$
for some $k\in\Z_{\geqslant 1}$. Since $4/k=a_2\in\Z_{>0}$, one
obtains that $k\in\{1,2,4\}$.

If $k\in\{1,2\}$, then $\gcd(a_1, a_2, a_3)=2$, which is
impossible. We see that $k=4$. Then
$$
\big(a_0, a_1, a_2, a_3,d\big)=\big(2s-1, 2s-2, 1, 2, 4s-2\big).
$$
\end{example}

\begin{example}
\label{example:8} Suppose that the~hypersurface $X$ is
well-formed, and suppose that the~quasihomogeneous polynomial
$\xi(x,y,z,w)$ is found in the~first part
of~\cite[Case~V.3(4)]{YauYu03}. Then\footnote{Note that there is
a~misprint in \cite[Case~V.3(4)]{YauYu03} and one should read $(r,
s)=(3, s)$ instead of $(s, r)=(3, s)$, and the~same correction
should be made in the~second and the~third part of this subcase.}
$$
\xi\big(x,y,z,w\big)\in\Big\{yx^3+xy^3+z^2+zw^s,\ yx^3+xy^3+z^s+zw^2\Big\},%
$$
where  $s\in\mathbb{Z}_{\geqslant 3}$. If
$\xi(x,y,z,w)=yx^3+xy^3+z^2+zw^s$, then
$$
3a_0+a_1=a_0+3a_1=2a_2=a_2+sa_3
$$
which contradicts the~well-formedness of the~hypersurface $X$.

We have $\xi(x,y,z,w)=yx^3+xy^3+z^s+zw^2$. Then
$3a_0+a_1=a_0+3a_1=sa_2=a_2+2a_3$ and
$$
\big(a_0, a_1, a_2, a_3,d\big)=\left(\frac{sa}{4}, \frac{sa}{4}, a, \frac{(s-1)a}{2}, sa\right),%
$$
where $a_2=a$. Since $\gcd(a_0, a_1, a_2)=1$, we see that $4\mid
a$. Then $a\in\{2,4\}$.

Suppose that $a=2$. Then
$$
\big(a_0, a_1, a_2, a_3,d\big)=\left(\frac{s}{2}, \frac{s}{2}, 2, s-1, 2s\right),%
$$
where $s$ is divisible by $2$ and not divisible by $4$. Then
$s=4n+2$, where $n\in\mathbb{Z}_{\geqslant 1}$. We have
$$
\big(a_0, a_1, a_2, a_3,d,I\big)=\big(2n+1, 2n+1, 2, 4n+1, 8n+4,1\big).%
$$

Suppose that $a=4$. Then $(a_0, a_1, a_2, a_3,d)=(s, s, 4,
2s-2,4s)$. Then
$$
\big(a_0, a_1, a_2, a_3,d,I\big)=\big(2n+1, 2n+1, 4, 4n,8n+4,2\big)%
$$
for some $n\in\mathbb{Z}_{\geqslant 1}$, because $s$ must be odd.
\end{example}

\begin{example}
\label{example:9} Suppose that the~hypersurface $X$ is
well-formed, and suppose that the~quasihomogeneous polynomial
$\xi(x,y,z,w)$ is found in the~first part
of~\cite[Case~XI.3(14)]{YauYu03}. Then
$$
\xi\big(x,y,z,w\big)=x^3+xy^3+yz^s+zw^2,
$$
where $s\in\mathbb{Z}_{\geqslant 3}$. Hence
$3a_0=a_0+3a_1=a_1+sa_2=a_2+2a_3$. Put $a_2=a$. Then
$$
\big(a_0, a_1, a_2, a_3,d\big)=\left(\frac{3as}{7}, \frac{2as}{7}, a, \frac{a(9s-7)}{14}, \frac{9as}{7}\right),%
$$
and $\gcd(a_0, a_1, a_2)=1$, because $X$ is well-formed. Thus
either $a=1$ or $a=7$.

Suppose that $a=1$. Then
$$
\big(a_0, a_1, a_2, a_3,d\big)=\left(\frac{3s}{7}, \frac{2s}{7}, 1, \frac{9s-7}{14},\frac{9s}{7}\right),%
$$
which implies that $s=7k$ for some $k\in\mathbb{Z}_{\geqslant 1}$.
Hence, we have
$$
\big(a_0, a_1, a_2, a_3,d\big)=\left(3k, 2k, 1, \frac{9k-1}{2}, 9k\right),%
$$
which implies that $k=2n-1$ for some $n\in\mathbb{Z}_{\geqslant
1}$. We have
$$
\big(a_0, a_1, a_2, a_3,d,I\big)=\big(6n-3, 4n-2, 1, 9n-5, 18n-9, n\big).%
$$

Suppose that $a=7$. Then
$$
\big(a_0, a_1, a_2, a_3,d\big)=\left(3s, 2s, 7, \frac{9s-7}{2},9s\right),%
$$
which implies that $s=2k+1$ for some $k\in\mathbb{Z}_{\geqslant
1}$. Hence, we have
$$
\big(a_0, a_1, a_2, a_3,d\big)=\big(6k+3, 4k+2, 7, 9k+1, 18k+9\big),%
$$
but $\gcd(a_0, a_1, a_2)=1$. Then $k\neq 3 \mod 7$. Thus, we have
 the~following solutions:\footnote{Note that in the resulting tables
of Appendix~\ref{section:infinite-series} we split the first of the obtained
series into a sporadic case corresponding to $n=1$ and a shifted
series starting from $n=2$. This is done to ensure that
$a_0\le\ldots\le a_3$.}
\begin{itemize}
\item $(a_0, a_1, a_2, a_3,d,I)=(28n-22, 42n-33, 7, 63n-53, 126n-99, 7n-2)$,%
\item $(a_0, a_1, a_2, a_3,d,I)=(28n-18, 42n-27, 7, 63n-44, 126n-81, 7n-1)$,%
\item $(a_0, a_1, a_2, a_3,d,I)=(28n-10, 42n-15, 7,63n-26, 126n-45, 7n+1)$,%
\item $(a_0, a_1, a_2, a_3,d,I)=(28n-6, 42n-9, 7, 63n-17, 126n-27, 7n+2)$,%
\item $(a_0, a_1, a_2, a_3,d,I)=(28n-2, 42n-3, 7, 63n-8, 126n-9, 7n+3)$,%
\item $(a_0, a_1, a_2, a_3,d,I)=(28n+2, 42n+3, 7, 63n+1, 126n+9, 7n+4)$,%
\end{itemize}
where $n\in\mathbb{Z}_{\geqslant 1}$.
\end{example}

\begin{example}
\label{example:10} Suppose that the~hypersurface $X$ is
well-formed, and suppose that the~quasihomogeneous polynomial
$\xi(x,y,z,w)$ is found in~\cite[Case~VIII.2(1)]{YauYu03}. Then
$$
\xi\big(x,y,z,w\big)=x^2+y^2+yz^r+yw^s+z^{\epsilon}w^{\zeta},
$$
where $r\in\mathbb{Z}_{\geqslant 2}\ni s$. Hence
$2a_0=2a_1=a_1+ra_2=a_1+sa_3$. Put $a_2=a$. Then
$$
\big(a_0, a_1, a_2, a_3,d\big)=\left(ra, ra, a, \frac{ra}{s}, 2ra\right),%
$$
where $a=\gcd(a_0, a_1, a_2)=1$, because $X$ is well-formed. Thus,
we have
$$
\big(a_0, a_1, a_2, a_3,d\big)=\left(r, r, 1, \frac{r}{s},2r\right),%
$$
where $r/s=\gcd(a_0, a_1, a_3)=1$. Then $(a_0, a_1, a_2,
a_3,d)=(r, r, 1, 1,2r)$.
\end{example}

\begin{example}
\label{example:11} Suppose that the~hypersurface $X$ is
well-formed, and suppose that the~quasihomogeneous polynomial
$\xi(x,y,z,w)$ is found in \cite[Case~XIV.1(1)]{YauYu03}. Then
$$
\xi\big(x,y,z,w\big)=x^r+xy+xz^s+xw^t+y^{\epsilon}z^{\zeta}+z^{\eta}w^{\theta},
$$
where $r, s, t\in\mathbb{Z}_{\geqslant 2}$. Hence
$ra_0=a_0+a_1=a_0+sa_2=a_0+ta_3$. Put $a_0=a$. Then
$$
\big(a_0, a_1, a_2, a_3,d\big)=\left(a, (r-1)a, \frac{(r-1)a}{s}, \frac{(r-1)a}{t},ra \right),%
$$

It follows from the~well-formedness of the~hypersurface $X$ that
$$
\gcd\big(a_0, a_1, a_2\big)=\gcd\big(a_0, a_1, a_3\big)=1,
$$
so that $a$ divides $s$ and~$t$. Put $s=ap$ and $t=aq$ for some
$q\in\mathbb{Z}_{\geqslant 1}\ni p$. Then
$$
\gcd\left(\frac{r-1}{p}, \frac{r-1}{q}\right)=1,
$$
because $\gcd(a_1, a_2, a_3)=1$, where $r-1$ is divisible by $p$
and $q$. Thus, we see that
$$
p=mk,\ q=ml,\ r-1=mkl,
$$
where $m$, $k$ and $l$ are positive integers such that $\gcd(k,
l)=1$. Then
$$
\big(a_0, a_1, a_2, a_3,d\big)=\big(a, mkla, l, k,(mkl+1)a\big)
$$

By well-formedness one obtains that $d$ is divisible by $\gcd(a_1,
a_2)=l$. Then $l\mid a$ and
$$
l\mid \gcd\big(a_0, a_1, a_2\big)
$$
so that by well-formedness $l=1$. In a~similar way we get $k=1$.
Then
$$
\big(a_0, a_1, a_2, a_3,d,I\big)=\big(a, ma, 1, 1, (m+1)a,2\big),
$$
where $m$ and $a$ are arbitrary positive integers.
\end{example}

\medskip

The~proof of Theorem~\ref{theorem:technical} is similar in
the~remaining cases.

\section{Bishop vs Lichnerowicz}
\label{section:obstructions}

In this section, we prove
Theorem~\ref{theorem:Bishop-versus-Lichnerowicz}. Let
$\bar{a}_{0},\ldots,\bar{a}_{n},\bar{d}$ be positive real numbers
such that
$$
0<\sum_{i=0}^{n}\bar{a}_{i}-\bar{d}\leqslant n\bar{a}_{0}
$$
and
$\bar{a}_{0}\leqslant\bar{a}_{1}\leqslant\ldots\leqslant\bar{a}_{n}$,
where $n\geqslant 1$. To prove
Theorem~\ref{theorem:Bishop-versus-Lichnerowicz}, we must show
that
$$
\bar{d}\Bigg(\sum_{i=0}^{n}\bar{a}_{i}-\bar{d}\Bigg)^{n}\leqslant
n^n\prod_{i=0}^{n}\bar{a}_{i}.
$$

Put $\bar{I}=\sum_{i=0}^{n}\bar{a}_{i}-\bar{d}$. Then $I=\alpha
n\bar{a}_0$, where $\alpha\in\mathbb{R}$ such that
$0<\alpha\leqslant 1$. We must prove that
\begin{equation}
\label{eqation:obstructions}
\Bigg(\sum_{i=1}^{n}\bar{a}_{i}+\big(1-\alpha n\big)\bar{a}_0\Bigg)\bar{a}_0^{n-1}\alpha^n-\prod_{i=1}^{n}\bar{a}_{i}\leqslant 0.%
\end{equation}

Put $a_{i}=\bar{a}_{i}/\bar{a}_{0}$ for every
$i\in\{1,\ldots,n\}$. Then $(\ref{eqation:obstructions})$ is
equivalent to
\begin{equation}
\label{eq:alpha} \Bigg(\sum_{i=1}^{n}a_i+1-\alpha n\Bigg)\alpha^n-\prod_{i=1}^{n}a_i\leqslant 0,%
\end{equation}
where $a_1\geqslant 1,a_2\geqslant 1,\ldots,a_n\geqslant 1$. But
to prove $(\ref{eq:alpha})$ is enough to prove that
\begin{equation}
\label{eq:no-alpha} \sum_{i=1}^{n}a_i+1-n-\prod_{i=1}^{n}a_i\leqslant 0,%
\end{equation}
because the~derivative of the~left hand side of~$(\ref{eq:alpha})$
with respect to $\alpha$ equals
$$
n\alpha^{n-1}\Bigg(\sum_{i=1}^{n}a_i+1-\alpha\big(n+1\big)\Bigg)\geqslant n\alpha^{n-1}\Bigg(\sum_{i=1}^{n}a_i-n\Bigg)\geqslant 0,%
$$
since $\alpha\leqslant 1$ and $a_i\geqslant 1$  every
$i\in\{1,\ldots,n\}$. Let us prove~$(\ref{eq:no-alpha})$ by
induction on $n$.

We may assume that $n\geqslant 2$, and  $a_{i}\ne 1$ for every
$i\in\{1,\ldots,n\}$ by the~induction assumption.

\begin{lemma}
\label{lemma:obstructions-bound} Suppose that $a_i\geqslant n$ for
some $i\in\{1,\ldots,n\}$. Then the~inequality~\ref{eq:no-alpha}
holds.
\end{lemma}

\begin{proof}
Without loss of generality, we may assume that $a_n\geqslant n$.
Then
$$
\sum_{i=1}^{n}a_i+1-n-\prod_{i=1}^{n}a_i=\sum_{i=1}^{n-1}\Bigg(a_i-\prod_{i=1}^{n-1}a_i\Bigg)+\big(a_n-n+1\big)\Bigg(1-\prod_{i=1}^{n-1}a_i\Bigg)\geqslant 0%
$$
which completes the~proof.
\end{proof}

Put $F(x_1,\ldots, x_n)=\sum_{i=1}^{n}x_i+1-n-\prod_{i=1}^{n}x_i$.
Let $U\subset\mathbb{R}^{n}$ be an~open set given by
$$
1<a_1<n,\ 1<a_2<n,\ldots, 1<a_n<n,
$$
and  suppose that $(\ref{eq:no-alpha})$ fails. Then
$F(a_{1},\ldots,a_{n})>0$. But
$$
\big(x_1,\ldots, x_n\big)\in\overline{U}\setminus U\Longrightarrow F\big(x_{1},\ldots,x_{n})\leqslant 0,%
$$
which implies that $F$ attains its maximum at some point $(A_1,
\ldots, A_{n})\in U$. Thus, we have
$$
A_{k}=\prod_{i=1}^{n}A_{i}
$$
for every $k\in\{1,\ldots,n\}$ by the~first derivative test.
The~latter implies $A_1=A_2=\ldots=A_n$.~Then
$$
nA_{1}+1-n-A_{1}^n>0,
$$
which is impossible, because $nA_{1}+1-n-A_{1}^n$ is a~decreasing
function  of $A_{1}$ vanishing at $A_{1}=1$.

The assertion of Theorem~\ref{theorem:Bishop-versus-Lichnerowicz}
is proved.

\appendix

\section{Tables}
\label{section:infinite-series}

Table~\ref{table:1} and Table~\ref{table:2}  contain
one-parameter
infinite series and sporadic cases respectively of~values of
$(a_0,a_1,a_2,a_3,d,I)$ in Theorem~\ref{theorem:technical}.
We always assume that $a_0\le\ldots\le a_3$.
The~last columns represent the~cases in~\cite{YauYu03} from which
the~sixtuples $(a_0,a_1,a_2,a_3,d,I)$ originate\footnote{Note that
sometimes a~sixtuple $(a_0,a_1,a_2,a_3,d,I)$ originates from
several cases in~\cite{YauYu03}.}. The parameter $n$ is any
positive integer.

\begin{longtable}{|c|c|c|c|}
\caption{Infinite series}\label{table:1}\\ \hline
$(a_0,a_1,a_2,a_3)$ & $d$ & $I$ & Source\\
\hline
\endhead

$(1, 3n-2, 4n-3, 6n-5)$ & $12n-9$ & $n$ & VII.2(3) \\
\hline

$(1, 3n-2, 4n-3, 6n-4)$ & $12n-8$ & $n$ & II.2(2) \\
\hline

$(1, 4n-3, 6n-5, 9n-7)$ & $18n-14$ & $n$ &VII.3(1) \\
\hline

$(1, 6n-5, 10n-8, 15n-12)$ & $30n-24$ & $n$ &  III.1(4) \\
\hline

$(1, 6n-4, 10n-7, 15n-10)$ & $30n-20$ & $n$ & III.2(2) \\
\hline

$(1, 6n-3, 10n-5, 15n-8)$ & $30n-15$ & $n$ & III.2(4) \\
\hline

$(1, 8n-2, 12n-3, 18n-5)$ & $36n-9$ & $2n$ & IV.3(3) \\
\hline

$(2, 6n-3, 8n-4, 12n-7)$ & $24n-12$ & $2n$ & II.2(4)\\
\hline

$(2, 6n+1, 8n+2, 12n+3)$ & $24n+6$ & $2n+2$ & II.2(1)\\
\hline

$(3, 6n+1, 6n+2, 9n+3)$ & $18n+6$ & $3n+3$ & II.2(1)\\
\hline

$(7, 28n-18, 42n-27, 63n-44)$ & $126n-81$ & $7n-1$ & XI.3(14)\\
\hline

$(7, 28n-17, 42n-29, 63n-40)$ & $126n-80$ & $7n+1$ & X.3(1)\\
\hline

$(7, 28n-13, 42n-23, 63n-31)$ & $126n-62$ & $7n+2$ & X.3(1)\\
\hline

$(7, 28n-10, 42n-15, 63n-26)$ & $126n-45$ & $7n+1$ & XI.3(14)\\
\hline

$(7, 28n-9, 42n-17, 63n-22)$ & $126n-44$ & $7n+3$ & X.3(1)\\
\hline

$(7, 28n-6, 42n-9, 63n-17)$ & $126n-27$ & $7n+2$ & XI.3(14)\\
\hline

$(7, 28n-5, 42n-11, 63n-13)$ & $126n-26$ & $7n+4$ & X.3(1)\\
\hline

$(7, 28n-2, 42n-3, 63n-8)$ & $126n-9$ & $7n+3$ & XI.3(14)\\
\hline

$(7, 28n-1, 42n-5, 63n-4)$ & $126n-8$ & $7n+5$ & X.3(1)\\
\hline

$(7, 28n+2, 42n+3, 63n+1)$ & $126n+9$ & $7n+4$ & XI.3(14)\\
\hline

$(7, 28n+3, 42n+1, 63n+5)$ & $126n+10$ & $7n+6$ & X.3(1)\\
\hline

$(7, 28n+6, 42n+9, 63n+10)$ & $126n+27$ & $7n+5$ & XI.3(14)\\
\hline

$(2, 2n+1, 2n+1, 4n+1)$ & $8n+4$ & $1$ & II.3(4) \\
\hline

$(3, 3n, 3n+1, 3n+1)$ & $9n+3$ & $2$ & III.5(1) \\
\hline

$(3, 3n+1, 3n+2, 3n+2)$ & $9n+6$ & $2$ & II.5(1) \\
\hline

$(3, 3n+1, 3n+2, 6n+1)$ & $12n+5$ & $2$ & XVIII.2(2)\\
\hline

$(3, 3n+1, 6n+1, 9n)$ & $18n+3$ & $2$ & VII.3(2)\\
\hline

$(3, 3n+1, 6n+1, 9n+3)$ & $18n+6$ & $2$ & II.2(2) \\
\hline

$(4, 2n+3, 2n+3, 4n+4)$ & $8n+12$ & $2$ &  V.3(4) \\
\hline

$(4, 2n+3, 4n+6, 6n+7)$ & $12n+18$ & $2$ & XII.3(17) \\
\hline

$(6, 6n+3, 6n+5, 6n+5)$ & $18n+15$ & $4$ & III.5(1) \\
\hline

$(6, 6n+5, 12n+8, 18n+9)$ & $36n+24$ & $4$ & VII.3(2)\\
\hline

$(6, 6n+5, 12n+8, 18n+15)$ & $36n+30$ & $4$ & IV.3(1) \\
\hline

$(8, 4n+5, 4n+7, 4n+9)$ & $12n+23$ & $6$ & XIX.2(2)\\
\hline

$(9, 3n+8, 3n+11, 6n+13)$ & $12n+35$ & $6$ & XIX.2(2)\\
\hline
\end{longtable}

\begin{longtable}{|c|c|c|c||c|c|c|c|}
\caption{Sporadic cases}\label{table:2}\\
\hline
$(a_0,a_1,a_2,a_3)$ & $d$ & $I$ & Source & $(a_0,a_1,a_2,a_3)$ & $d$ & $I$ & Source\\
\hline
\endhead

$(1, 3, 5, 8)$ & $16$ & $1$ & VIII.3(5) &
$(2, 3, 5, 9)$ & $18$ & $1$ & II.2(3)\\
\hline

$(3, 3, 5, 5)$ & $15$ & $1$ & I.19 &
$(3, 5, 7, 11)$ & $25$ & $1$ & X.2(3)\\
\hline

$(3, 5, 7, 14)$ & $28$ & $1$ & VII.4(4)&
$(3, 5, 11, 18)$ & $36$ & $1$ & VII.3(1)\\
\hline

$(5, 14, 17, 21)$ & $56$ & $1$ & XI.3(8)&
$(5, 19, 27, 31)$ & $81$ & $1$ & X.3(3)\\
\hline

$(5, 19, 27, 50)$ & $100$ & $1$ & VII.3(3)&
$(7, 11, 27, 37)$ & $81$ & $1$ & X.3(4)\\
\hline

$(7, 11, 27, 44)$ & $88$ & $1$ & VII.3(5)&
$(9, 15, 17, 20)$ & $60$ & $ 1$ & VII.6(3)\\
\hline

$(9, 15, 23, 23)$ & $69$ & $1$ & III.5(1) &
$(11, 29, 39, 49)$ & $127$ & $1$ & XIX.2(2)\\
\hline

$(11, 49, 69, 128)$ & $256$ & $1$ & X.3(1)&
$(13, 23, 35, 57)$ & $127$ & $1$ & XIX.2(2)\\
\hline

$(13, 35, 81, 128)$ & $256$ & $1$ & X.3(2)&
$(1, 3, 4, 6)$ & $12$ & $2$ & I.3 \\
\hline

$(1, 4, 6, 9)$ & $18$ & $2$ & IV.3(3) &
$(1, 6, 10, 15)$ & $30$ & $2$ & I.4\\
\hline

$(2, 3, 4, 7)$ & $14$ & $2$ & IX.3(1)&
$(3, 3, 4, 4)$ & $12$ & $2$ & V.3(4)\\
\hline

$(3, 4, 5, 10)$ & $20$ & $2$ & II.3(2)&
$(3, 4, 6, 7)$ & $18$ & $2$ & VII.3(10)\\
\hline

$(3, 4, 10, 15)$ & $30$ & $2$ & II.2(3)&
$(5, 13, 19, 22)$ & $57$ & $2$ & X.3(3)\\
\hline

$(5, 13, 19, 35)$ & $70$ & $2$ & VII.3(3)&
$(6, 9, 10, 13)$ & $36$ & $2$ & VII.3(8)\\
\hline

$(7, 8, 19, 25)$ & $57$ & $2$ & X.3(4)&
$(7, 8, 19, 32)$ & $64$ & $2$ & VII.3(3)\\
\hline

$(9, 12, 13, 16)$ & $48$ & $2$ & VII.6(2)&
$(9, 12, 19, 19)$ & $57$ & $2$ & III.5(1)\\
\hline

$(9, 19, 24, 31)$ & $81$ & $2$ & XI.3(20)&
$(10, 19, 35, 43)$ & $105$ & $2$ & XI.3(18)\\
\hline

$(11, 21, 28, 47)$ & $105$ & $2$ & XI.3(16)&
$(11, 25, 32, 41)$ & $107$ & $2$ & XIX.3(1)\\
\hline

$(11, 25, 34, 43)$ & $111$ & $2$ & XIX.2(2)&
$(11, 43, 61, 113)$ & $226$ & $2$ & X.3(1)\\
\hline

$(13, 18, 45, 61)$ & $135$ & $2$ & XI.3(14)&
$(13, 20, 29, 47)$ & $107$ & $2$ & XIX.3(1)\\
\hline

$(13, 20, 31, 49)$ & $111$ & $2$ & XIX.2(2)&
$(13, 31, 71, 113)$ & $226$ & $2$ & X.3(2)\\
\hline

$(14, 17, 29, 41)$ & $99$ & $2$ & XIX.2(3)&
$(5, 7, 11, 13)$ & $33$ & $3$ & X.3(3)\\
\hline

$(5, 7, 11, 20)$ & $40$ & $3$ & VII.3(3)&
$(11, 21, 29, 37)$ & $95$ & $3$ & XIX.2(2)\\
\hline

$(11, 37, 53, 98)$ & $196$ & $3$ & X.3(1)&
$(13, 17, 27, 41)$ & $95$ & $3$ & XIX.2(2)\\
\hline

$(13, 27, 61, 98)$ & $196$ & $3$ & X.3(2)&
$(15, 19, 43, 74)$ & $148$ & $3$ & X.3(1)\\
\hline

$(5, 6, 8, 9)$ & $24$ & $4$ & VII.3(2)&
$(5, 6, 8, 15)$ & $30$ & $4$ & IV.3(1)\\
\hline

$(9, 11, 12, 17)$ & $45$ & $4$ & XI.3(20)&
$(10, 13, 25, 31)$ & $75$ & $4$ & XI.3(14)\\
\hline

$(11, 17, 20, 27)$ & $71$ & $4$ & XIX.3(1)&
$(11, 17, 24, 31)$ & $79$ & $4$ & XIX.2(2)\\
\hline

$(11, 31, 45, 83)$ & $166$ & $4$ & X.3(1)&
$(13, 14, 19, 29)$ & $71$ & $4$ & XIX.3(1)\\
\hline

$(13, 14, 23, 33)$ & $79$ & $4$ & XIX.2(2)&
$(13, 23, 51, 83)$ & $166$ & $4$ & X.3(2)\\
\hline

$(6, 7, 9, 10)$ & $27$ & $5$ & XI.3(14)&
$(11, 13, 19, 25)$ & $63$ & $5$ & XIX.2(2)\\
\hline

$(11, 25, 37, 68)$ & $136$ & $5$ & X.3(1)&
$(13, 19, 41, 68)$ & $136$ & $5$ & X.3(2)\\
\hline

$(11, 19, 29, 53)$ & $106$ & $6$ & X.3(1)&
$(13, 15, 31, 53)$ & $106$ & $6$ & X.3(2)\\
\hline

$(11, 13, 21, 38)$ & $76$ & $7$ & X.3(1)&
&&&\\
\hline
\end{longtable}

\end{document}